\documentclass[10pt]{article}

\usepackage{amssymb}
\usepackage{amsfonts}
\usepackage{amsthm}
\usepackage{amsmath}
\usepackage{float}
\usepackage{bigints}
\usepackage{fullpage}

\usepackage[dvips]{graphicx}   
\usepackage{color,epsfig}      

\newtheorem{Lemma}{Lemma}

\newtheorem{Theorem}{Theorem}

\newtheorem{Definition}{Definition}
\newtheorem{Problem}{Problem}

\newtheorem{Corollary}{Corollary}
\newtheorem{Remark}{Remark}

\numberwithin{Subcase}{Case}

\newcommand{\Eu}{\mathbb{E}}

\newcommand{\N}{\mathbb{N}}

\newcommand{\BB}{\mathbf{B}}

\renewcommand{\P}{\mathcal{P}}
\newcommand{\F}{\mathcal{F}}

\DeclareMathOperator{\conv}{conv}

\DeclareMathOperator{\area}{area}
\DeclareMathOperator{\bd}{bd}

\DeclareMathOperator{\inter}{int}

\DeclareMathOperator{\cl}{cl}

\begin{document}

\title{Bounds for totally separable translative packings in the plane}

\author{K\'{a}roly Bezdek\thanks{Partially supported by a Natural Sciences and
Engineering Research Council of Canada Discovery Grant} and Zsolt L\'angi\thanks{Partially supported by the National Research, Development and Innovation Office, NKFI, K-119670}}
\date{}
 \maketitle

\begin{abstract}
A packing of translates of a convex domain in the Euclidean plane is said to be totally separable if any two packing elements can be separated by a line disjoint from the interior of every packing element. This notion was introduced by G. Fejes T\'{o}th and L. Fejes T\'{o}th (1973) and has attracted significant attention. In this paper we prove an analogue of Oler's inequality for totally separable translative packings of convex domains and then we derive from it some new results. This includes finding the largest density of totally separable translative packings of an arbitrary convex domain and finding the smallest area convex hull of totally separable packings (resp., totally separable soft packings) generated by given number of translates of a convex domain (resp., soft convex domain). Finally, we determine the largest covering ratio (that is, the largest fraction of the plane covered by the soft disks) of an arbitrary totally separable soft disk packing with given soft parameter.
\vspace{2mm}


\vspace{2mm}

\end{abstract}

\section{Introduction}\label{intro}

Our paper intends to bridge totally separable packings of discrete geometry and Oler's inequality of geometry of numbers. The concept of totally separable packings was introduced by G. Fejes T\'oth and L. Fejes T\'oth in \cite{FeFe} as follows. We say that a set of domains is totally separable if any two of them can be separated by a straight line avoiding all of the domains. The main question investigated in \cite{FeFe} is to find the densest totally separable arrangement of congruent replicas of a given domain. The paper \cite{FeFe} generated a good deal of interest in the density problem of totally separable arrangements and led to further important publications such as  \cite{Andras} and \cite{Ke}. Coming from this direction our goal was to find the densest totally separable arrangement of translates of a given domain and then to extend that approach to the analogue question for finite totally separable arrangements. It turned out that an efficient method to achieve all that is based on a new version of Oler's classical inequality (\cite{Oler}). So, next we introduce some basic terminology and then state Oler's inequality in the form which is most suitable for this paper.

Let $K$ be a {\it convex domain}, i.e., a compact convex set with non-empty interior in the Euclidean plane $\Eu^2$. A family $\mathcal{F}$ of $n$ translates of $K$ in $\Eu^2$ is called a {\it packing} if no two members of $\mathcal{F}$ have an interior point in common.

If $K$ is an $o$-symmetric convex domain in $\Eu^2$, where $o$ stands for the origin of $\Eu^2$, then let $| \cdot |_{K}$ denote the {\it norm generated by $K$}, i.e., let $|x|_{K} = \min\{ \lambda : x \in \lambda K\}$ for any $x\in \Eu^2$. The distance between the points $p$ and $q$ measured in the norm $| \cdot |_{K}$ is denoted by $|p-q|_K$. For the sake of simplicity, the Euclidean distance between the points $p$ and $q$ of $\Eu^2$ is denoted by $|p-q|$.

If $P = \bigcup_{i=1}^n [x_{i-1},x_i]$ is a polygonal curve in $\Eu^2$, and $K$ is an $o$-symmetric plane convex domain, then the {\it Minkowski length} of $P$ is defined as $M_K(P) = \sum_{i=1}^n |x_i-x_{i-1}|_K$. Based on this and using approximation by closed polygons one can define the Minkowski length $M_K(G)$ of any rectifiable curve $G \subseteq \Eu^2$ in the norm $| \cdot |_{K}$. If $K$ is a not $o$-symmetric, by $M_K(G)$ we mean the length of $G$ in the \emph{relative norm} of $K$, i.e., in the norm defined by $\frac{1}{2}(K-K)$ \cite{Oler}.

Finally, if $K$ is an $o$-symmetric convex domain in $\Eu^2$, then let $\diamond(K)$ denote a minimal area circumscribed hexagon of $K$.

Now, we are ready to state Oler's inequality (\cite{Oler}) in the following form. Let $K$ be an $o$-symmetric convex domain in $\Eu^2$. Let $$\mathcal{F} = \{ x_i + K : i=1,2,\ldots, n\}$$ be a packing of $n$ translates of $K$ in $\Eu^2$, and set $X = \{ x_1, x_2, \ldots, x_n\}$. Furthermore, let $\Pi$ be a simple closed polygonal curve with the following properties:  \begin{enumerate}
\item the vertices of $\Pi$ are points of $X$

and

\item $X \subseteq \Pi^*$ with $\Pi^* = \Pi \cup \inter \Pi$, where $\inter \Pi$ refers to the interior of $\Pi$.
\end{enumerate}
Then
\begin{equation}\label{eq:Oler-original}
\frac{\area (\Pi^*)}{\area \left(\diamond (K)\right)} + \frac{M_K(\Pi)}{4} + 1 \geq n,
\end{equation}
where $\area(\cdot)$ denotes the area of the corresponding set. The formula (\ref{eq:Oler-original}) was conjectured by H. J. Zassenhaus and has a number of interesting aspects discussed in \cite{Z} (see also \cite{BetkeHenkWills} and \cite{BoRu}).

The rest of the paper is organized as follows. First, we prove an analogue of Oler's inequality for totally separable translative packings of convex domains (Section~\ref{1}) and then we derive from it some new results. This includes finding the largest density of totally separable translative packings of an arbitrary convex domain (Section~\ref{2}) and finding the smallest area convex hull of totally separable packings (resp., totally separable soft packings) generated by given number of translates of a convex domain (resp., soft convex domain) (Sections~\ref{3} and~\ref{4}). Finally, we determine the largest covering ratio (that is, the largest fraction of the plane covered by the soft disks) of an arbitrary totally separable soft disk packing with given soft parameter (Section~\ref{5}).

\section{An analogue of Oler's inequality for totally separable translative packings}\label{1}

We need the following definitions. 

\begin{Definition}\label{defn:totallyseparable}
Let $K \subseteq \Eu^2$ be a convex domain. A packing $\mathcal{F}$ of translates of $K$ in $\Eu^2$ is called \emph{totally separable} if any two members of $\mathcal{F}$ can be separated by a line which is disjoint from the interiors of all members of $\mathcal{F}$.
\end{Definition}



\begin{Definition}\label{defn:permissiblepolygon}
A closed polygonal curve $P = \bigcup_{i=1}^m [x_{i-1},x_i]$, where $x_0 = x_m$, is called \emph{permissible} if there is a sequence
of simple closed polygonal curves $P^n = \bigcup_{i=1}^m [x^n_{i-1},x^n_i]$, where $x^n_0 = x^n_m$, satisfying $x^n_i \to x_i$ for every value of $i$. The interior $\inter P$ is defined as $\lim_{n \to \infty} \inter P^n$.
\end{Definition}

\begin{Remark}\label{rem:interioriswelldefined}
By the properties of limits, if $P = \bigcup_{i=1}^m [x_{i-1},x_i]$ is permissible and $P^n$ and $Q^n$ are sequences of simple closed polygonal curves with $\lim_{n \to \infty} P^n = \lim_{n \to \infty} Q^n = P$, then $\lim_{n \to \infty} \inter P^n = \lim_{n \to \infty} \inter Q^n$, i.e., the interior of a permissible curve is well defined.
\end{Remark}

\begin{Definition}
Let $K$ be an $o$-symmetric convex domain in $\Eu^2$. Then let $\square(K)$ denote a minimal area circumscribed parallelogram of $K$.
\end{Definition}

The main result of this section is the following totally separable analogue of Oler's inequality.

\begin{Theorem}\label{thm:Oler}
Let $K$ be an $o$-symmetric convex domain in $\Eu^2$. Let $$\mathcal{F} = \{ x_i + K : i=1,2,\ldots, n\}$$ be a totally separable packing of $n$ translates of $K$ in $\Eu^2$, and set $X = \{ x_1, x_2, \ldots, x_n\}$. Furthermore, let $\Pi$ be a permissible closed polygonal curve with the following properties:  \begin{enumerate}
\item the vertices of $\Pi$ are points of $X$

and

\item $X \subseteq \Pi^*$ with $\Pi^* = \Pi \cup \inter \Pi$.
\end{enumerate}
Then
\begin{equation}\label{eq:Oler}
\frac{\area (\Pi^*)}{\area \left(\square (K)\right)} + \frac{M_K(\Pi)}{4} + 1 \geq n.
\end{equation}
\end{Theorem}

\begin{Remark}\label{equality}
We note that equality in (\ref{eq:Oler}) of Theorem~\ref{thm:Oler} is attained in a variety of ways as indicated in Fig.~\ref{fig:Oler}, which consists of blocks of zig-zags and simple closed polygons having sides parallel to the two sides of a chosen $\square (K)$.
\end{Remark}

\begin{figure}[ht]
\begin{center}
\includegraphics[width=0.45\textwidth]{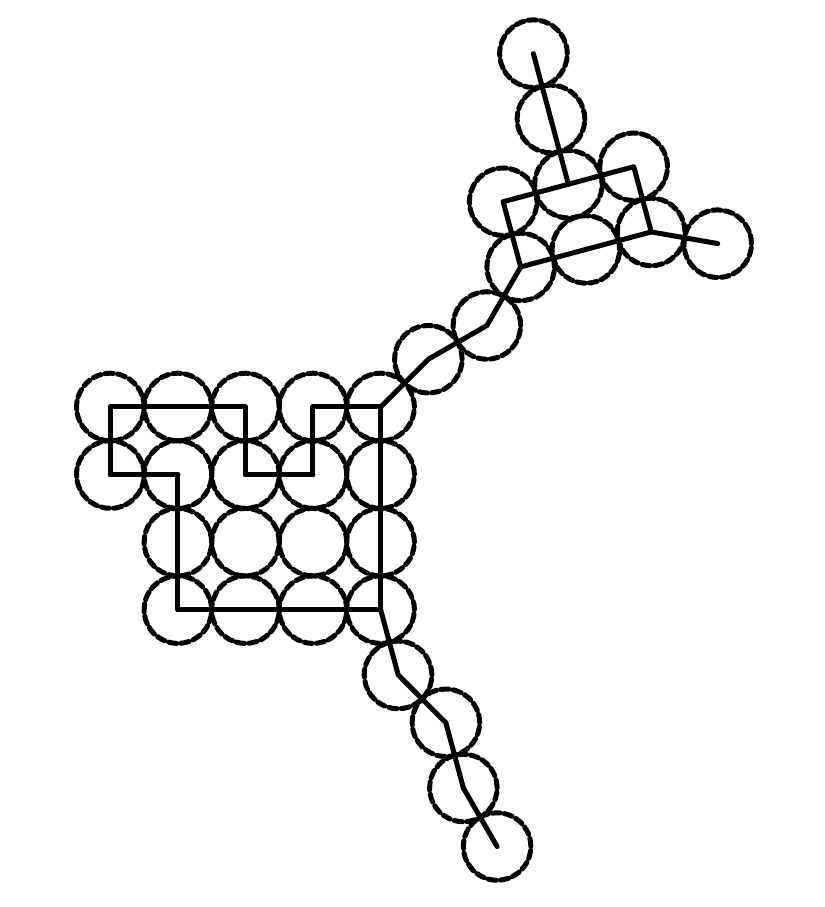}
\caption{A totally separable packing of translates of $K$ (with $K$ being a circular disk for the sake of simplicity), which satisfies the conditions in Theorem~\ref{thm:Oler} and for which there is equality in (\ref{eq:Oler}) of Theorem~\ref{thm:Oler}.}
\label{fig:Oler}
\end{center}
\end{figure}

\begin{Remark}\label{rem:parallelogram}
It is well-known that the width of any convex body $K$ in any direction is equal to the width of its central symmetrization $\frac{1}{2}(K-K)$ in this direction. This readily implies that $\square(K)$ does not change under central symmetrization.
\end{Remark}

\begin{Remark}\label{no-symmetry}
Let $\mathcal{F} = \{ x_i + K : i=1,2,\ldots, n\}$ be a family of $n$ translates of $K$ in $\Eu^2$, where $K$ is an $o$-symmetric convex domain of $\Eu^2$, and let $K^*$ be a convex domain satisfying $K = \frac{1}{2} (K^*-K^*)$ with $o \in \inter K^*$, and let $\mathcal{F}^* = \{ x_i + K^* : i=1,2,\ldots, n\}$. Then $\mathcal{F}$ is a packing if and only if $\mathcal{F}^*$ is a packing, and $\mathcal{F}$ is a totally separable packing if and only if $\mathcal{F^*}$ is a totally separable packing. (For details see for example, \cite{BeKhOl}.)
Thus, Theorem~\ref{thm:Oler} holds for any (not necessarily $o$-symmetric) plane convex domain $K^*$ (with $o \in \inter K^*$) as well.
\end{Remark}

\begin{Remark}
In fact, the proof of Theorem~\ref{thm:Oler} presented in this section works for more general point sets $X$ as well. Namely, one may assume only
that $\Pi^*$ can be cut into $n$ pieces by $(n-1)$ successive cuts by segments (cutting only one piece at a time) such that
\begin{itemize}
\item each segment starts (resp., ends) at some point of $\Pi$ or a preceding segment;
\item the relative interior of each segment is contained in the piece it cuts;
\item when cutting a piece by a segment then no point of $X$ lying in the given piece lies closer to this segment than one measured in the norm generated by $K$;
\item at the end, each piece contains exactly one point of $X$.
\end{itemize}
\end{Remark}

\begin{proof}
For any permissible closed polygonal curve $\Pi$ in Theorem~\ref{thm:Oler}, we set
\[
F(\Pi) = \frac{\area (\Pi^*)}{\area \left(\square (K)\right)} + \frac{M_K(\Pi)}{4} + 1.
\]

We prove the assertion by induction on $n$. Clearly, if $n=1$, then $F(\Pi) = 0 + 0 + 1=1$, and Theorem~\ref{thm:Oler} holds.

Assume that for any $n' < n$, Theorem~\ref{thm:Oler} holds for any totally separable translative packing of $K$ with $n'$ elements and for any permissible polygonal curve associated to it.
We prove that it holds for $n$ element packings as well.

Let $L$ be a line intersecting $\Pi$ and  separating the elements of $\mathcal{F}$.
We present the proof for the case only that $L$ intersects $\Pi$ at exactly two points, as the proof in the other cases is similar.
Let these intersection points be $p$ and $q$. Then $p$ and $q$ are points in the relative interior of some edges $p \in [p_1,p_2]$ and $q \in [q_1,q_2]$ of $\Pi$ whose vertices are not contained in $L + \inter K$. For simplicity, we imagine $L$ as a horizontal line, $p$ to the left of $q$, and $p_1$ and $q_1$ to be above $L$. Let $L_1$ and $L_2$ be the upper, respectively lower, line bounding $L + K$. For $i=1,2$, let $p'_i$ and $q'_i$ be the intersection points of $L_i$ with $[p,p_i]$ and $[q,q_i]$, respectively.

Without loss of generality, we assume that the parallelogram $P_L$ circumscribed about $K$ and having the property that its area is minimal among the circumscribed parallelograms, under the condition that it has a pair of sides parallel to $L$, is a square of edge length $2$. Thus, we have $\area (P_L) = 4$.

Observe that the lines $L_1$ and $L_2$ decompose $\Pi$ into four components: one above $L_1$, one below $L_2$, and the last two ones being the segments $[p'_1,p'_2]$ and $[q'_1,q'_2]$. We define $\Pi'_1$ as the union of the component above $L_1$ and the segment $[p'_1,q'_1]$, and we define $\Pi'_2$ similarly. Clearly, these polygonal curves are permissible. Finally, for $i=1,2$, we let $\Pi'^*_i = \Pi'_i \cup \inter \Pi'_i$ (cf. Figure~\ref{fig:Thm1_1}).

\begin{figure}[ht]
\begin{center}
\includegraphics[width=0.5\textwidth]{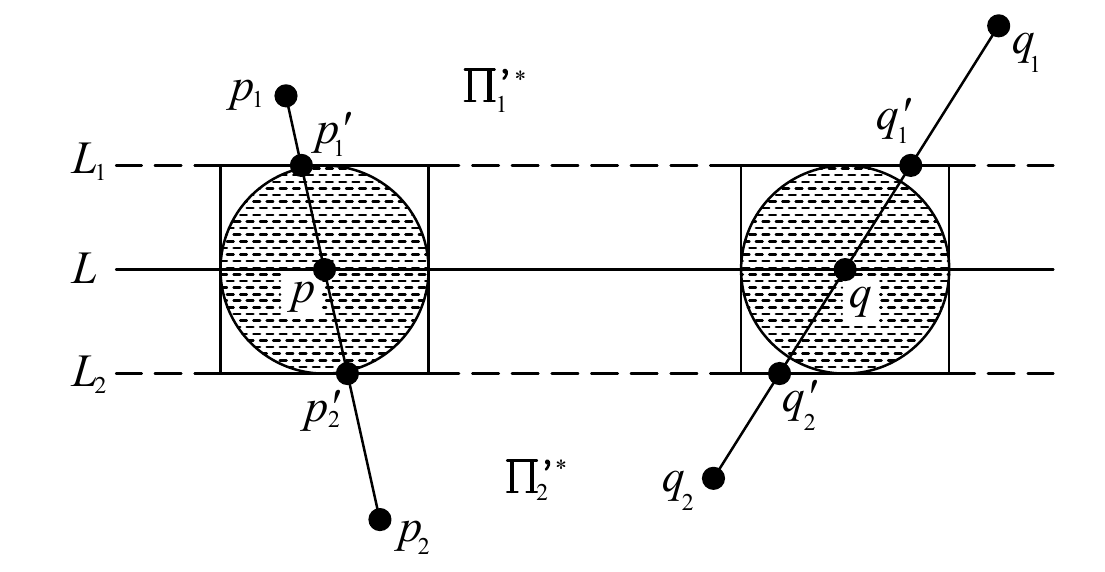}
\caption{Notations in the proof of Theorem~\ref{thm:Oler}}
\label{fig:Thm1_1}
\end{center}
\end{figure}

Then we have
\begin{equation}\label{eq:area}
\area (\Pi^*) = \area(\Pi'^*_1)+\area(\Pi'^*_2)+ \area \left( \conv \{ p'_1,p'_2, q'_2, q'_1\} \right) =
\end{equation}
\[
\area(\Pi'^*_1)+\area(\Pi'^*_2) + 2 |p-q| = \area(\Pi'^*_1)+\area(\Pi'^*_2) + 2 |p-q|_K.
\]
Furthermore, since the normed distance of $L_1$ and $L_2$ is two, we have
\begin{equation}\label{eq:perimeter}
M_K(\Pi) = M_K(\Pi'_1)+M_K(\Pi'_2)-|p'_1-q'_1|_K - |p'_2-q'_2|_K + |p'_1-p'_2|_K + |q'_1-q'_2|_K \geq
\end{equation}
\[
\geq M_K(\Pi'_1)+M_K(\Pi'_2) -2|p-q|_K +4.
\]
Now we define a polygonal curve in $\Pi'^*_1$. Consider the points of $X$ in the region $R_1$ bounded by $[p_1,p'_1]$, $[p'_1,q'_1]$, $[q'_1,q_1]$ and $[q_1,p_1]$. Note that this region is a (not necessarily convex) quadrangle. If $R_1$ is not convex, we assume, without loss of generality, that $p_1 \in \conv \{p'_1,q_1,q'_1\}$. Consider the ray starting at $p_1$ through $p$ and begin to rotate it counterclockwise until it hits the first point $p_2$ of $X$ in $R_1$. Then rotate this half line about $p_2$ clockwise until it hits the next point of $X$. Continuing
this process we end up with a simple curve $C_1$ in $R_1$, starting at $p_1$ and ending at $q_1$, which divides $R_1$ into two connected components one of which contains all points of $X$ in $R_1$. We remark that if $R_1$ is convex, then $C_1$ is a convex curve.

Let $\Pi_1$ denote the closed polygonal curve $\left( \Pi'_1 \setminus ([p_1,p'_1] \cup [p'_1,q'_1] \cup [q'_1,q_1] ) \right) \cup C_1$. It is easy to see that $\Pi_1$ is a permissible polygonal curve whose vertices are points of $X$ above $L$, and whose interior contain every other point of $X$ above $L$. Let $\Pi^*_1 = \Pi_1 \cup \inter \Pi_1$.
Clearly, $\area (\Pi^*_1) \leq \area (\Pi'^*_1)$. We show that $M_K(\Pi_1) \leq M_K(\Pi'_1)$.

\emph{Case 1:} $R_1$ is convex.\\
Note that in this case $C_1 \cup [p_1,q_1]$ is a convex region contained in the convex region $R_1$, and thus, $M_K(C_1)+|p_1-q_1|_K \leq M_K (\bd R_1)$, which readily implies our claim.

\begin{figure}[ht]
\begin{center}
\includegraphics[width=0.5\textwidth]{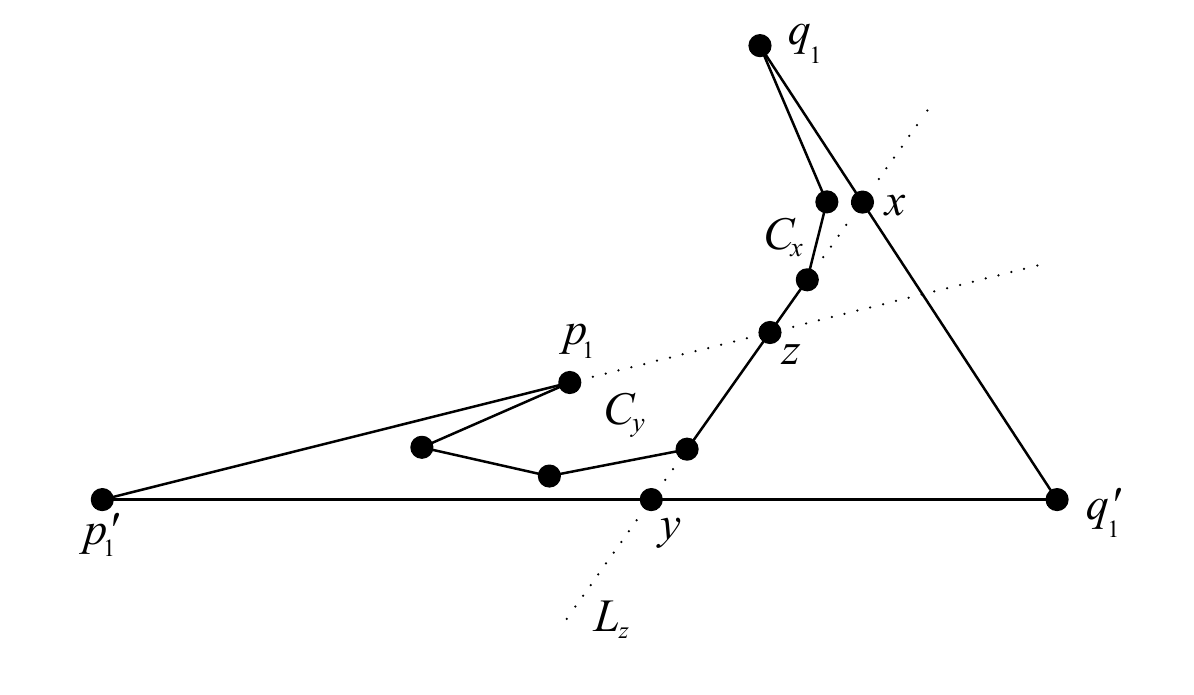}
\caption{The points $p_1,p'_1, q_1, q'_1$ are not in convex position as in Case 2}
\label{fig:Thm1_2}
\end{center}
\end{figure}

\emph{Case 2:} $R_1$ is not convex.\\
Then, according to our assumption, the line through $p'_1$ and $p_1$ intersects $[q_1,q'_1]$ (cf. Figure~\ref{fig:Thm1_2}). This line intersects $C_1$ at exactly one point $z$, and there is a line $L_z$ through $z$ which supports $C_1$ at $z$. Let $L_z$ intersect $[q_1,q'_1]$ at $x$, and $[p'_1,q'_1]$ at $y$.
The point $z$ decomposes $C_1$ into two convex polygonal curves $C_x$ and $C_y$ such that $p_1 \in C_y$ and $q_1 \in C_x$.
Then we have
\[
M_K(C_1) = M_K(C_y) + M_K(C_x) \leq |p_1-p'_1|_K + | p'_1-y|_K + |y-z|_K + |z-x|_K +|x-q_1|_K \leq
\]
\[
\leq |p_1-p'_1|_K + |p'_1 - q'_1|_K + |q'_1-q_1|_K.
\]
This implies that $M_K(\Pi_1) \leq M_K(\Pi'_1)$.

To construct a permissible polygon $\Pi_2$ in $\Pi'^*_2$ with the same properties we may apply an analogous process.

Thus, we have obtained two permissible polygons $\Pi_1$ and $\Pi_2$ associated to totally separable translative packings of $K$, with strictly less elements than $n$, say $k$ and $n-k$. Now, by (\ref{eq:area}), (\ref{eq:perimeter}), $\area \left(\square (K)\right) \leq 4$, $\area (\Pi^*_i) \leq \area (\Pi'^*_i)$, $M_K(\Pi_i) \leq M_K(\Pi'_i)$ (for $i=1,2$),  and the induction hypothesis, we have
\[
F(\Pi) \geq \frac{\area(\Pi^*_1)+\area(\Pi^*_2) + 2 |p-q|_K}{\area \left(\square (K)\right)} + \frac{M_K(\Pi_1)+M_K(\Pi_2) -2|p-q|_K +4}{4} + 1 \geq
\]
\[
\geq F(\Pi_1) + F(\Pi_2) \geq k+n-k = n.
\]
\end{proof}

\section{On the densest totally separable translative packings}\label{2}

Theorem~\ref{thm:Oler} and Remark~\ref{no-symmetry} imply the following statement, which was proved (using a method different from ours) for $o$-symmetric convex domains in \cite{FeFe} with a weaker estimate than (\ref{Fejes-Toth}) for convex domains in general namely, with $\square(K)$ standing for a minimal area circumscribed quadrangle of $K$.

\begin{Theorem}\label{cor:FeFe}
If $\delta_{sep}(K)$ denotes the largest (upper) density of totally separable translative packings of the convex domain $K$ in $\Eu^2$, then
\begin{equation}\label{Fejes-Toth}
\delta_{sep}(K) = \frac{\area(K)}{\area(\square(K))}.
\end{equation}
\end{Theorem}

\begin{proof}
Clearly, $\delta_{sep}(K) \geq \frac{\area(K)}{\area(\square(K))}$. To show the opposite inequality, without loss of generality we may assume that $o \in \inter K$.

Set $C = \min \{ \mu>0: K-K \subseteq \mu K\}$, and $C' = \frac{M_K(\bd K)}{4}$.
Consider any totally separable packing $\mathcal{F}$ of translates of $K$ in $\Eu^2$. For any $t > 0$, let $\mathcal{F}_t$ denote the subfamily of $\mathcal{F}$ consisting of the elements that intersect $tK$, and let $X_t$ denote the set of the translation vectors of the elements of $\mathcal{F}_t$ and $n_t$ the cardinality of $\mathcal{F}_t$.
Note that if $y \in (x+K) \cap tK$, then $x+K \subseteq y+(K-K)$ and therefore $x+K  \subseteq (t+C)K$, implying that $\bigcup \mathcal{F}_t \subseteq (t+C)K$.
On the other hand, by Theorem~\ref{thm:Oler} and Remark~\ref{no-symmetry}, it follows that
\[
n_t \leq \frac{\area (\conv (X_t))}{\area(\square (K))}+\frac{M_K(\bd \conv(X_t))}{4}+1 \leq
\]
\[
\frac{\area ((t+C)K)}{\area(\square (K))}+\frac{M_K(\bd ((t+C)K))}{4}+1 = (t+C)^2 \frac{\area (K)}{\area(\square (K))}+(t+C)C'+1.
\]
This yields that
\[
\frac{\area((\bigcup \mathcal{F}) \cap tK)}{\area(tK)} \leq \frac{\area(\bigcup\mathcal{F}_t)}{\area(tK)}=\frac{n_t\area (K)}{\area(tK)}=\frac{n_t}{t^2} \leq \frac{(t+C)^2\area(K)}{t^2\area(\square(K))} + \frac{tC'+CC'+1}{t^2},
\]
from which the claim follows by letting $t\to+\infty$.
\end{proof}

Theorem~\ref{cor:Fary} is a totally separable analogue of the well-known theorem (which is a combination of the results published in \cite{Fary}, \cite{FTL50}, \cite{FTL83}, and \cite{R}), stating that the maximal density of translative packings of a convex domain in $\Eu^2$ is minimal if and only if the domain is a triangle.

\begin{Theorem}\label{cor:Fary}
For any convex domain $K$ in $\Eu^2$, we have
\begin{equation}\label{eq:density}
\frac{1}{2} \leq \delta_{sep}(K) \leq 1,
\end{equation}
with equality on the left if and only if $K$ is a triangle, and on the right if and only if $K$ is a parallelogram.
\end{Theorem}

\begin{proof}
The right-hand side inequality in (\ref{eq:density}) is an immediate consequence of Theorem~\ref{cor:FeFe}. We prove only the left-hand side inequality.
Let $P$ be a minimum area parallelogram circumscribed about $K$. Without loss of generality, we may assume that $P$ is the square $[0,1]^2$ in a suitable Cartesian coordinate system.
Let the sides of $P$ be $S_1, S_2, S_3, S_4$ in counterclockwise order such that the endpoints of $S_1$ are $(0,0)$ and $(1,0)$. Since $P$ has minimum area, each side of $P$ intersects $K$.

We show that $S_2 \cap K$ and $S_4 \cap K$ contain points with equal $y$-coordinates.
Suppose for contradiction that it is not so. Then $(1,0)+(S_4 \cap K)$ and $S_2 \cap K$ are disjoint, implying that there is some point $p_2 \in S_2$ separating these two sets. Set $p_4 = (-1,0)+p_2$. Then we may rotate the line of $S_2$ around $p_2$, and the line of $S_4$ around $p_4$ slightly, with the same angle, to obtain a parallelogram containing $K$, with area equal to $\area(P)$ and having two sides disjoint from $K$, which contradicts our assumption that $P$ has minimum area. Thus, there are some points $p_4 = (0,t) \in S_4 \cap K$ and $p_2=(1,t)\in S_2 \cap K$ for some $t \in [0,1]$.
We obtain similarly the existence of points $p_1 = (s,0) \in S_1 \cap K$ and $p_3=(s,1) \in S_3 \cap K$.
Hence, $\area(K) \geq \area (\conv \{ p_1,p_2,p_3,p_4\}) = \frac{1}{2} \area(P)$, which yields the left-hand side inequality in (\ref{eq:density}).

Now we examine the equality case. Note that, using the notations of the previous paragraph, $\frac{1}{2} = \delta_{sep}(K)$ implies that
$K = \conv \{ p_1,p_2,p_3,p_4\}$. Consider the case that $s,t \in (0,1)$. Let $P'$ be the parallelogram obtained by rotating the line of $S_2$ around $p_2$ and the line of $S_4$ around $p_4$, with the same small angle. Then $P'$ is a parallelogram circumscribed about $K$, having area equal to $\area(P)$.
Let the sides of $P'$ be $S'_1,S'_2,S'_3, S'_4$ such that for $i=1,2,3,4$, $p_i \in S'_i$. Observe that $S'_1 \cap K = \{ p_1\}$, $S'_3 \cap K = \{p_3\}$, and $[p_1,p_3]$ is not parallel to $S'_2$ and $S'_4$. Thus, applying the argument in the previous paragraph, it follows that $P'$ is not a minimum area circumscribed parallelogram, a contradiction.
Thus, $s$ or $t$ is equal to $0$ or $1$, which implies that $K$ is a triangle.
\end{proof}

\section{On the smallest area convex hull of totally separable translative finite packings}\label{3}

\begin{Theorem}\label{thm:areaformula}
Let $\mathcal{F} = \{ c_i + K : i=1,2,\ldots, n\}$ be a totally separable packing of $n$ translates of the convex domain $K$ in $\Eu^2$.
Let  $C = \conv \{ c_1,c_2,\ldots, c_n \}$.
\begin{enumerate}
\item[(\ref{thm:areaformula}.1)]
Then we have
\[
\area \left(\conv \left( \bigcup_{i=1}^n (c_i+K) \right)\right) = \area (C+K)\geq \frac{2}{3} (n-1)\area \left(\square(K)\right) + \area (K)+\frac{1}{3} \area(C).
\]
\item[(\ref{thm:areaformula}.2)]
If $K$ or $C$ is centrally symmetric, then
\[
\area (C+K)\geq (n-1) \area \left(\square(K)\right) + \area (K).
\]
\end{enumerate}
\end{Theorem}

\begin{Remark}
We note that equality is attained in $(\ref{thm:areaformula}.1)$ of Theorem~\ref{thm:areaformula} for the following totally separable translative packings of a triangle (cf. Figure~\ref{fig:area_triangle}).
Let $K$ be a triangle, with the origin $o$ at a vertex, and $u$ and $v$ being the position vectors of the other two vertices, and let $T = m K$, where $m > 1$ is an integer. Let $\mathcal{F}$ be the family consisting of  the elements of the lattice packing $\{ iu+jv +K : i,j, \in \mathbb{Z} \}$ contained in $T$.
Then $\mathcal{F}$ is a totally separable packing of $n=\frac{m(m+1)}{2}$ translates of $K$ with $\conv \left(\bigcup \mathcal{F} \right) = T=C+K$, where $C=(m-1)K$.
Thus, $\area(T) = m^2 \area (K)=[\frac{2}{3}m(m+1)-\frac{1}{3}+\frac{1}{3}(m-1)^2]\area (K)=\frac{4}{3}(n-1)\area (K)+\area (K)+\frac{1}{3}\area (C)=$
$\frac{2}{3}(n-1)\area (\square(K))+\area(K)+\frac{1}{3}\area(C)$.
\end{Remark}

\begin{figure}[ht]
\begin{center}
\includegraphics[width=0.2\textwidth]{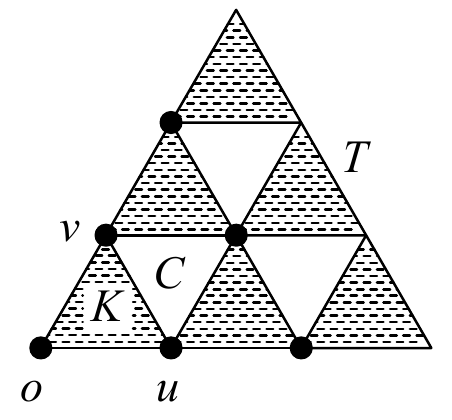}
\caption{An example for equality in (\ref{thm:areaformula}.1)}
\label{fig:area_triangle}
\end{center}
\end{figure}

\begin{Remark}
In (\ref{thm:areaformula}.2) of Theorem~\ref{thm:areaformula} equality can be attained in a variety of ways shown in Figure~\ref{fig:convexhull} for both cases namely, when $C$ is centrally symmetric (and $K$ is not centrally symmetric such as a triangle) and when $K$ is centrally symmetric (such as a circular disk) without any assumption on the symmetry of $C$. 
\end{Remark}


\begin{figure}[ht]
\begin{center}
\includegraphics[width=0.6\textwidth]{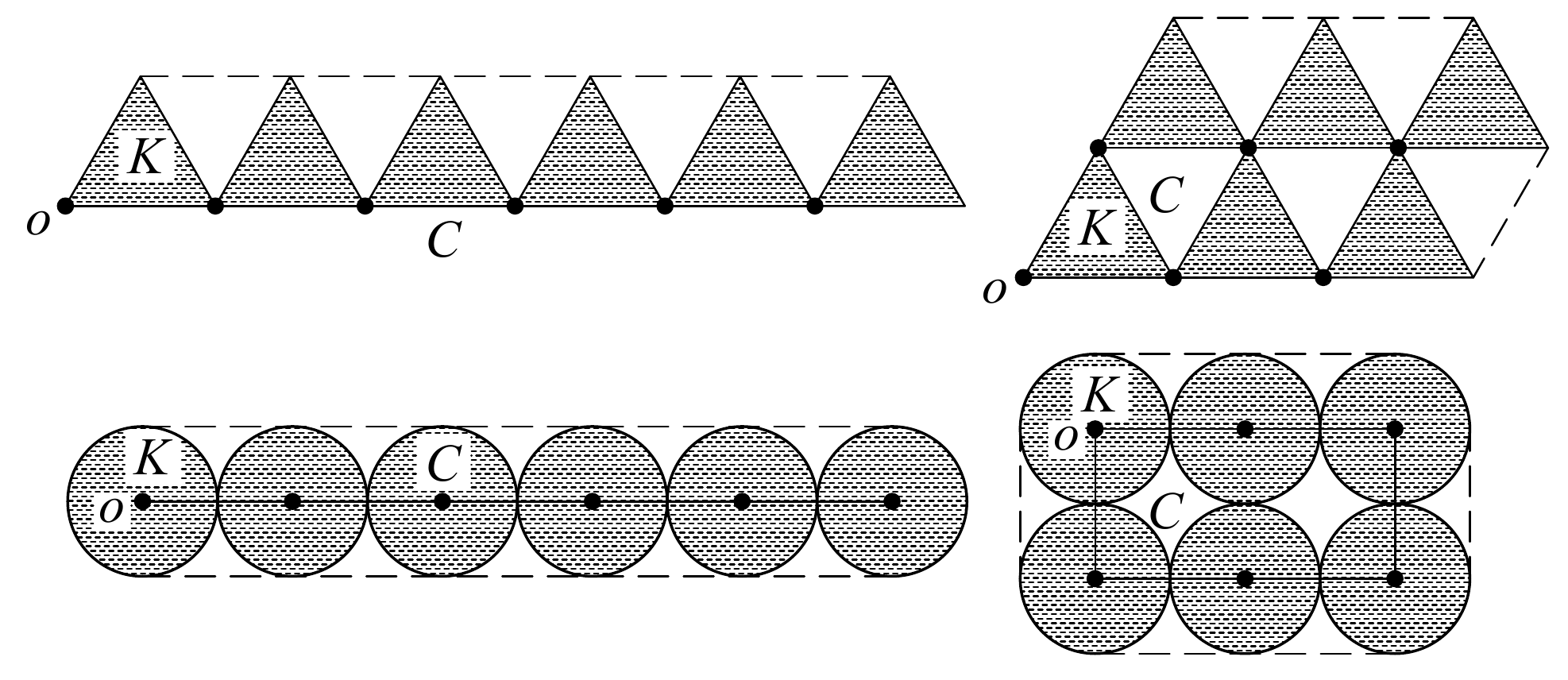}
\caption{Totally separable translative packings of a triangle and a unit disk for which equality is attained in (\ref{thm:areaformula}.2) in Theorem~\ref{thm:areaformula}}
\label{fig:convexhull}
\end{center}
\end{figure}

\begin{proof}
We start with proving the following inequalities.
\begin{Lemma}\label{lem:Radon}
Let $K$ be a convex domain in $\Eu^2$ and let $Q$ be a convex polygon. Furthermore, let $A(Q,K)$ denote the mixed area of $Q$ and $K$.
\begin{enumerate}
\item[(\ref{lem:Radon}.1)] Then we have
\[
\frac{12 A(Q,K)}{\area\left(\square (K)\right)} \geq M_K(\bd Q).
\]
Here, equality holds, for instance, if $Q=K$ is a triangle.
\item[(\ref{lem:Radon}.2)] If $K$ or $Q$ is centrally symmetric, then
\[
\frac{8 A(Q,K)}{\area\left(\square (K)\right)} \geq M_K(\bd Q).
\]
Furthermore, if $K$ is centrally symmetric, then equality holds for every convex polygon $Q$ if and only if $\bd K$ is a Radon curve.
\end{enumerate}
\end{Lemma}


\begin{proof}
Without loss of generality, we may assume that $\area(\square(K))=1$.

Let $k$ denote the number of sides of $Q$, and for $i=1,2,\ldots, k$, let $l_i$ and $x_i$ denote the (Euclidean) length and the outer unit normal vector of the $i$th side of $Q$. Note that then for every value of $i$, $w_K(x_i) = h_K(x_i)+h_K(-x_i)$, where $w_K(x_i)$ is the width of $K$ in the direction of $x_i$ and $h_K(x)=\sup \{x\cdot k\ :\ k\in K\}$ is the support function of $K$ evaluated at $x\in\Eu^2$ with "$\cdot$" standing for the standard inner product of $\Eu^2$.
Furthermore, observe that
\begin{equation}\label{eq:mixedarea}
A(Q,K)= \frac{1}{2} \sum_{i=1}^k l_i h_K(x_i).
\end{equation}

First, we prove (\ref{lem:Radon}.2) for the case that $K$ is centrally symmetric. Since a translation of $K$ or $Q$ does not change their mixed area, we may assume that $K$ is $o$-symmetric, which implies that $h_K(x_i)= \frac{1}{2} w_K(x_i)$ for every $i$. Let $r_i$ be the Euclidean length of the radius of $K$ in the direction of the $i$th side of $Q$. Then the normed
length of this side is $\frac{l_i}{r_i}$. On the other hand, since $2r_iw_K(x_i)$ is the area of a parallelogram circumscribed about $K$ having minimum area under the condition that it has a side parallel to the $i$th side of $Q$, therefore for every value of $i$ we have $r_i w_K(x_i) \geq \frac{1}{2}$. Combining these observations and (\ref{eq:mixedarea}), it follows that
\[
A(Q,K) = \frac{1}{4} \sum_{i=1}^k l_i w_K(x_i) = \frac{1}{4} \sum_{i=1}^k \frac{l_i}{r_i} r_i w_K(x_i) \geq \frac{1}{8} \sum_{i=1}^k \frac{l_i}{r_i} = \frac{1}{8} M_K(Q).
\]
Here, equality holds for every convex polygon $Q$ if and only if for any $v \in \mathbb{S}^1=\{x\in\Eu^2\ :\ |x|=1\}$, there is a minimum area parallelogram circumscribed about $K$, which has a side parallel to $v$. In other words,
for any $v \in \mathbb{S}^1$, we have that $l_K(v) w_K(v^{\perp})$ is independent of $v$, where $l_K(v)$ is the length of a longest chord of $K$ in the direction of $v$, and $w_K(v^{\perp})$ is the width of $K$ in the direction perpendicular to $v$.
The observation that this property is equivalent to the fact that $\bd K$ is a Radon curve can be found, for example, in the proof of Theorem 2 of \cite{GHorvathLangi}.

Now consider the case that $Q$ is $o$-symmetric, but $K$ is not necessarily. Note that in this case $k$ is even, and for every $i$ we have
$l_{i+k/2} = l_i$, and $x_{i+k/2}=-x_i$. Thus, by (\ref{eq:mixedarea})
\[
A(Q,K) = \frac{1}{2}\sum_{i=1}^{k/2}l_i(h_K(x_i)+h_K(-x_i)) = \frac{1}{2} \sum_{i=1}^{k/2} l_i w_K(x_i) = \frac{1}{4} \sum_{i=1}^k l_i w_K(x_i).
\]
From this equality, the statement follows by a similar argument using the relative norm of $K$ whenever $K$ is not centrally symmetric.

Finally we prove (\ref{lem:Radon}.1) about the general case.
Let $\bar{K}= \frac{1}{2}(K-K)$. Without loss of generality, we may assume that the origin $o$ is the center of a maximum area triangle inscribed in $K$. Then, clearly, $-\frac{1}{2}K \subseteq K$, from which a simple algebraic transformation yields that $\frac{2}{3}\bar{K} \subseteq K$.
This implies that for any unit vector $x$ we have
\begin{equation}\label{eq:asymmRadon}
h_K(x) \geq \frac{2}{3} h_{\bar{K}}(x).
\end{equation}
Then, by (\ref{eq:mixedarea}), we have
\[
A(Q,K) \geq  \frac{1}{3} \sum_{i=1}^k l_i h_{\bar{K}}(x_i) = \frac{2}{3} A(Q,\bar{K}).
\]
Thus, our inequality readily follows from (\ref{lem:Radon}.2). The fact that here equality holds if $Q=K$ is a triangle can be shown by an elementary computation.
\end{proof}

First, we prove (\ref{thm:areaformula}.2).
Note that $\bd C$ satisfies the conditions in Theorem~\ref{thm:Oler}, and thus (using Remark~\ref{no-symmetry} if $K$ is not centrally symmetric), we have
\[
\frac{\area (C)}{\area \left(\square (K)\right)} + \frac{M_K(\bd C)}{4} + 1  \geq n.
\]
Thus, $(\ref{lem:Radon}.2)$ of Lemma~\ref{lem:Radon} yields that
\[
\frac{\area (C)}{\area \left(\square (K)\right)} + \frac{2 A(C,K)}{\area \left(\square (K)\right)} + 1  \geq n.
\]
From this, it follows that
\[
\area \left(\conv \left( \bigcup_{i=1}^n (c_i+K) \right)\right)= \area (C+K)=\area (C) + 2 A(C,K) + \area (K)\geq
\]
\[
(n-1) \area \left(\square (K)\right) + \area (K).
\]

Now we prove (\ref{thm:areaformula}.1).
In this case, Theorem~\ref{thm:Oler} applied to $\bd C$ in the same way as above followed by $(\ref{lem:Radon}.1)$ of Lemma~\ref{lem:Radon} implies that
\[
(n-1) \area(\square(K)) \leq \area(C) + 3A(C,K) = \frac{3}{2} \area(C+K) - \frac{1}{2} \area(C) - \frac{3}{2} \area(K).
\]
This inequality yields
\[
\area(C+K) \geq \frac{2(n-1)}{3} \area(\square(K)) + \area(K) + \frac{1}{3} \area(C),
\]
finishing the proof of (\ref{thm:areaformula}.1).
\end{proof}

\section{On the smallest area convex hull of totally separable translative finite soft packings}\label{4}

The following notion has been defined for Euclidean balls in $\Eu^d$ in \cite{BezdekLangi}.

\begin{Definition}
Let $K$ be an $o$-symmetric convex domain in $\Eu^2$. Let $\lambda \geq 0$, and let $K^{\lambda}$ denote the soft domain $(1+\lambda) K$ with the soft parameter $\lambda$, hard core $K$, and soft annulus $(1+\lambda)K \setminus K$ in $\Eu^2$.
\end{Definition}

\begin{Remark}
Clearly, $K^{\lambda}$ and $K^{\lambda} \setminus K$ are symmetric about $o$ in $\Eu^2$.
\end{Remark}

\begin{Definition}
Let $\{ c_1, c_2, \ldots, c_n \} \subseteq \Eu^2$. We say that $\{ c_1 + K^{\lambda}, c_2 + K^{\lambda}, \ldots, c_n + K^{\lambda} \}$
is a \emph{totally separable soft packing} of $n$ translates of the soft domain $K^{\lambda}$ in $\Eu^2$, if
$\{ c_1 + K, c_2+K, \ldots, c_n + K\}$ is a totally separable packing in the usual sense (see Definition \ref{defn:totallyseparable}). Let $\P^{sep}_{K,n,\lambda}$ be the family of all totally separable soft packings of $n$ translates of the soft domain $K^{\lambda}$ for given $K$, $n > 1$, $\lambda \geq 0$.
\end{Definition}

The following statement is an extension of (\ref{thm:areaformula}.2) of Theorem~\ref{thm:areaformula} and also it is a totally separable version of Theorem 2.1 in \cite{BetkeHenkWills}.

\begin{Theorem}\label{thm:Betke}
Let $K$ be an $o$-symmetric convex domain in $\Eu^2$, and let $n > 1$ and $\lambda \geq 0$ be given.
If $\{ c_1 + K^{\lambda}, c_2 + K^{\lambda}, \ldots, c_n + K^{\lambda} \} \in \P^{sep}_{K,n,\lambda}$, then
\[
\area \left( \conv \left( \bigcup_{i=1}^n (c_i + K^{\lambda}) \right) \right) \geq
\]
\[
(n-1)\area \left(\square(K)\right)+2\lambda A(\conv \{c_1,c_2,\ldots,c_n \}, K)+(1+\lambda)^2 \area (K) \geq
\]
\[
\left( n-1 + \frac{\lambda}{4} M_K\left(\bd(\conv \{ c_1,c_2,\ldots, c_n\})\right) \right) \area \left(\square(K)\right) + (1+\lambda)^2 \area(K).
\]
\end{Theorem}

\begin{Remark}
We note that equality in Theorem~\ref{thm:Betke} is attained, for example, for "sausages" in the form $\{ c_i + K : i=1,2,\ldots, n\}$, where $|c_2-c_1|_K=\dots =|c_n-c_{n-1}|_K=2$, and $c_2-c_1,\dots , c_n-c_{n-1}$ are parallel to a chosen side of $\square (K)$.
\end{Remark}

\begin{proof}
First, we prove the inequality stated first. By the definition of mixed area (cf. e.g. \cite{BetkeHenkWills}), we have
\[
\area\left(\conv \left( \bigcup_{i=1}^n (c_i + K)\right)\right) =\area(\conv \{c_1,c_2,\ldots,c_n \}+K)=
\]
\[
\area(\conv \{c_1,c_2,\ldots,c_n \})+2A(\conv \{c_1,c_2,\ldots,c_n \}, K) + \area(K).
\]
Thus, (\ref{thm:areaformula}.2) of Theorem~\ref{thm:areaformula} implies that
\begin{equation}\label{eq:convexhullarea}
\area(\conv \{c_1,c_2,\ldots,c_n \}) \geq (n-1)\area(\square (K))-2A(\conv \{c_1,c_2,\ldots,c_n \}, K).
\end{equation}
Again using the definition of mixed area and also (\ref{eq:convexhullarea}) we get that
\[
\area\left(\conv \left( \bigcup_{i=1}^n (c_i + K^{\lambda})\right)\right) =\area(\conv \{c_1,c_2,\ldots,c_n \}+K^{\lambda})=
\]
\[
\area(\conv \{c_1,c_2,\ldots,c_n \}) +2(1+\lambda) A(\conv \{c_1,c_2,\ldots,c_n \}, K) + (1+\lambda)^2 \area(K)\geq
\]
\[
(n-1)\area(\square (K))+2\lambda A(\conv \{c_1,c_2,\ldots,c_n \}, K)+(1+\lambda)^2 \area (K)
\]
finishing the proof of the first inequality. Finally, (\ref{lem:Radon}.2) of Lemma~\ref{lem:Radon} implies the second inequality of Theorem~\ref{thm:Betke} in a straightforward way.
\end{proof}

\bigskip

\section{On the covering ratio of totally separable soft disk packings}\label{5}

Let $B$ denote the circular disk of radius $1$ (in short, the unit disk) centered at $o$ in $\Eu^2$.

\begin{Definition}
Let $\F = \{ c_i + B : c_i\in\Eu^2\ {\text for}\ i \in \N\}$ be a totally separable packing (resp., lattice packing) of unit disks in $\Eu^2$. Then $\F^{\lambda} = \{ c_i + (1+\lambda)B : i \in \N\}$
is called a \emph{totally separable soft packing} (resp.,  \emph{totally separable soft lattice packing}) of the \emph{soft disks}  $c_i + (1+\lambda)B$ each being congruent to the soft disk $B^{\lambda}=(1+\lambda)B$ with \emph{soft parameter} $\lambda > 0$.
In this case the \emph{(upper) covering ratio} of the soft packing $\F^{\lambda}$ is defined as 
\[
\rho(\F^\lambda) = \limsup_{r \to \infty} \frac{\area \left( r B \cap \bigcup_{i\in\N } (c_i+B^{\lambda})\right)}{\area (r B)}.
\]
We denote by $\rho_{\lambda, B}^{sep}$ (respectively, $\rho_{\lambda, B}^{sep, lattice}$) the supremum of the (upper) covering ratios over the family of totally separable soft packings (respectively, of totally separable soft lattice packings) of soft disks congruent to $B^{\lambda}$ with soft parameter $\lambda > 0$.
\end{Definition}

We note that in \cite{BezdekLangi} the covering ratio just introduced was called soft density. We prefer to use the term covering ratio in order to emphasize that it means the fraction of plane covered by the soft elements of the given soft packing.
To state our main result in this section, for $\frac{\pi}{6} \leq \alpha \leq \frac{\pi}{4}$, we denote by $T_{\alpha}$  an isosceles triangle whose half angle at its apex is $\alpha$, and the two heights starting at the endpoints of its base are equal to two (cf. Figure~\ref{fig:Talpha}). Observe that if the vertices of this triangle are $p_1,p_2,p_3$, then the triple $\{ p_1+B, p_2 + B, p_3+B\}$ is totally separable, and the vectors $p_2-p_1$, $p_3-p_1$ generate a totally separable lattice packing of translates of $B$.
We introduce the notation

\[
\rho(\lambda,T_{\alpha}) = \frac{\area \left( T_{\alpha} \cap \bigcup_{i=1}^3 (p_i+ B^{\lambda} ) \right)}{\area(T_{\alpha})}.
\]

\begin{figure}[ht]
\begin{center}
\includegraphics[width=0.4\textwidth]{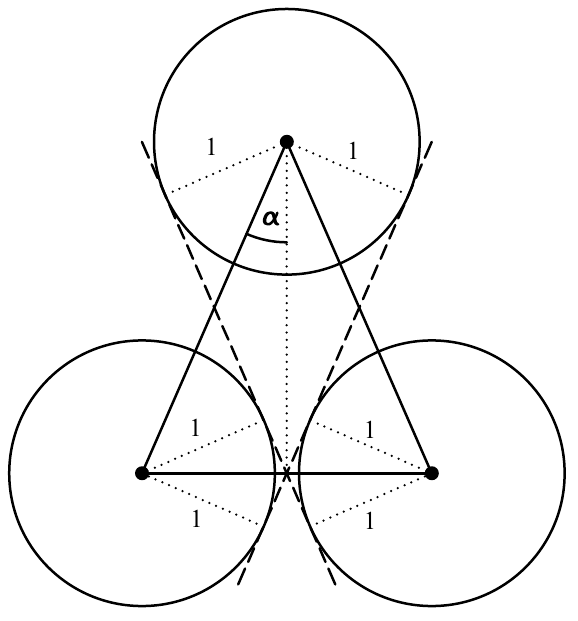}
\caption{The isosceles triangle $T_{\alpha}$ for some value of $\alpha$}
\label{fig:Talpha}
\end{center}
\end{figure}

\begin{Theorem}\label{thm:soft_packing}
For every $\lambda > 0$, we have $\rho_{\lambda,B}^{sep}=\rho_{\lambda,B}^{sep,lattice} = \rho(\lambda,T_{\alpha})$ for some $\frac{\pi}{6} \leq \alpha \leq \frac{\pi}{4}$.
\end{Theorem}

The proof is based on a refinement of a tessellation defined by Moln\'ar in \cite{Molnar}.
Let $C=\{c_i : i\in\N\}\subseteq \Eu^2$ be a \emph{saturated} point set in $\Eu^2$, that is, assume that there are some values $0 < a \leq b$ such that the distance of any two points in $C$ is at least $a$, and for any $x \in \Eu^2$, $|x-c_i| < b$ for some $c_i \in C$. For any $c_i \in C$, let the \emph{Voronoi cell} $D_i$ of $c_i$ be the set of points in $\Eu^2$ not farther from $c_i$ than from any other element of $C$. If for any cells $D_i$ and $D_j$ meeting at an edge, this edge is replaced by the segment $[c_i,c_j]$, we obtain the \emph{Delaunay} tessellation of $\Eu^2$. 
In this tessellation, the circumcenter of any cell is a vertex of some Voronoi cells, and the circumcircle of any Delaunay cell contains no point of $C$ in its interior. (Moreover, the circumcircle of any Delaunay cell contains no point of $C$ different from the vertices of the cell.)

Let the circumcenter of a cell $P$ of the Delaunay decomposition be $v$. If the line through an edge $S = [c_i,c_j]$ of $P$ separates $P$ and $v$, we say that $S$ is a \emph{separating side} of $P$. Then the polygonal curve $[c_i,v] \cup [v,c_j]$ is called the \emph{bridge} of $P$. Clearly, every cell $P$ has at most one bridge.
Let us replace the separating side of each cell (if it exists), by the bridge of the cell. Then, by Lemma 1 of \cite{Molnar}, we obtain another cell decomposition of $\Eu^2$, which we call \emph{Moln\'ar} tessellation or in short, $M$-tessellation.

\begin{proof}
We prove the theorem for a larger family of packings which we call \emph{weakly separable} packings of unit disks: we assume only that any three unit disks in the packing $\F$, under the condition that the pairwise distances between their centers are at most $2\sqrt{2}$, form a totally separable triple.
Observe that if $\F = \{ c_i + B : c_i \in C\}$ is a weakly separable packing of unit disks, and there is some point $p \in \Eu^2$ such that $|p-c_i| > 2\sqrt{2}$ for all $c_i\in C$, then after adding the circle $p + \BB$ to the packing it remains weakly separable.
Thus, we may assume that $C$ is saturated, and the circumradius of any Voronoi cell of $C$ is at most $2\sqrt{2}$.
In the proof we let $\F^{\lambda} = \{ c_i + B^{\lambda} : c_i \in C\}$, and for any region $Q$ in the plane, $\rho(\F^{\lambda} | Q) = \frac{\area(Q \cap \bigcup \F^{\lambda})}{\area(Q)}$. 

Let $P$ be an arbitrary Delaunay cell of $C$. Assume that the circumradius of $P$ is less than $\sqrt{2}$. Then, since the sides of $P$ are at least $2$, $P$ is an acute triangle, which, thus, contains its circumcenter. Hence, if $P$ has a separating side $S$, then it separates the Delaunay cell $P'$, meeting $P$ in $S$, from the circumcenter $v'$ of $P'$. On the other hand, since the circumcircle of $P'$ does not contain vertices of $P$ in its interior, it follows that the circumradius of $P'$ is not greater than that of $P$. This yields that the circumradius of $P'$ is less than $\sqrt{2}$, which contradicts our assumption that a side of $P'$ separates $v'$ from $P'$. In other words, we have that if the circumradius of $P$ is less than $\sqrt{2}$, then it is a triangle which remains the same in the $M$-tessellation as well.

We show that any other $M$-cell can be decomposed into cells of the form
\[
Q = \cl \left( \conv \{ v,c_i,c_j\} \setminus \conv \{ v',c_i,c_j\} \right),
\]
where $c_i, c_j \in C$, $|v-c_i| = |v-c_j| \geq \sqrt{2}$, and $|v'-c_i| = |v'-c_j|$.
Let $P$ be a Delaunay cell of $C$ with circumradius at least $\sqrt{2}$. Assume, first, that $P$ contains its circumcenter $v$. Thus, if $[c_i,c_j]$ is a separating side, then it separates the cell $P'$ which meets $P$ in $[c_i,c_j]$, from its circumcenter $v'$.
Then $[c_i,c_j]$ is replaced by the bridge $[c_i,v'] \cup [v',c_j]$. Note that $\sqrt{2}\leq |v'-c_i| = |v'-c_j| < |v-c_i|=|v-c_j| $. If $[c_i,c_j]$ is not a separating side, then one can choose $v'$ to be the midpoint of $[c_i,c_j]$. Thus, dissecting the $M$-cell obtained from $P$ by the segments connecting $v$ to the vertices of $P$ results in regions with the desired property.
If $P$ does not contain its circumcenter, we may apply a similar construction.
We call this tessellation the \emph{refined} $M$-tessellation, or $M'$-tessellation. Then, if $P$ is an $M'$-cell, then $P$ is either
\begin{itemize}
\item[(i)] an acute triangle with circumradius less than $\sqrt{2}$, in this case we say that $P$ is type 1, or
\item[(ii)] it is of the form $P=\cl \left( \conv \{ v,c_i,c_j\} \setminus \conv \{ v',c_i,c_j\} \right)$, where $c_i,c_j \in C$, $v$ is the circumcenter of a Delaunay cell with $c_i$ and $c_j$ as vertices and with circumradius at least $\sqrt{2}$, and $|v'-c_i| = |v'-c_j|$. In this case we say that $P$ is type 2.
\end{itemize}

\begin{Lemma}\label{lem:Molnarisgood}
Let $P$ be an $M'$-cell defined by $C$. Then, for any point $p \in \inter P$, if $c_i \in C$ is closest to $p$, then $c_i$ is a vertex of $P$.
\end{Lemma}

\begin{proof}
Consider the case that $P$ is type 1. Let $c_i \in C$ be closest to $p$.
Let $P'$ be a \emph{Delaunay cell} with $c_i$ as a vertex such that $P'$ intersects $[p,c_i]$. If $[p,c_i]$ does not intersect $\inter P'$, then it is contained in a sideline of $P'$. On the other hand, since $c_i$ is closest to $p$ in $C$, and no side of $P'$ crosses the edges of $P$, this is impossible. Thus, $[p,c_i]$ intersects $\inter P'$, which means that the line $L$ through two other vertices $c_j$, $c_k$ of $P'$ separates $p$ from $P'$. Let $x$ and $y$ denote the intersection points of $L$ with the circle, centered at $p$, of radius $|c_i-p|$. Since $|c_j-p|, |c_k-p| \geq |c_i-p|$, we have $[x,y] \subseteq [c_j,c_k]$. Note that since $L$ separates $c_i$ and $p$, it follows that $\frac{\pi}{2} < \angle(x,c_i,y) \leq \angle(c_j,c_i,c_k)$; that is, $P'$ has an obtuse angle at $c_i$. Hence, if $v'$ denotes the circumcenter of $P'$, then $L$ separates $v'$ and $P'$, or in other words, $[c_j,v'] \cup [c_k,v']$ is a bridge. Since no bridge can cross the sides of $P$,
to finish the proof it suffices to show that $p \in \conv \{ c_j,c_k,v'\}$. Let $L_j$ (respectively, $L_k$) be the line bisecting the segment $[c_i,c_j]$ (respectively, $[c_i,c_k]$). Note that $L_j$ and $L_k$ intersect at $v'$. Let $V$ be the closed convex angular region, with apex $v'$ and $\bd V \subseteq L_j \cup L_k$ such that $c_i \in V$. Note that since $p$ is not farther from $c_i$ than from $c_j$ or $c_k$, we have $p \in V$, which yields that $p \in \conv \{ c_j,c_k,v'\}$ (cf. Figure~\ref{fig:closestpoint}).

\begin{figure}[ht]
\begin{center}
\includegraphics[width=0.35\textwidth]{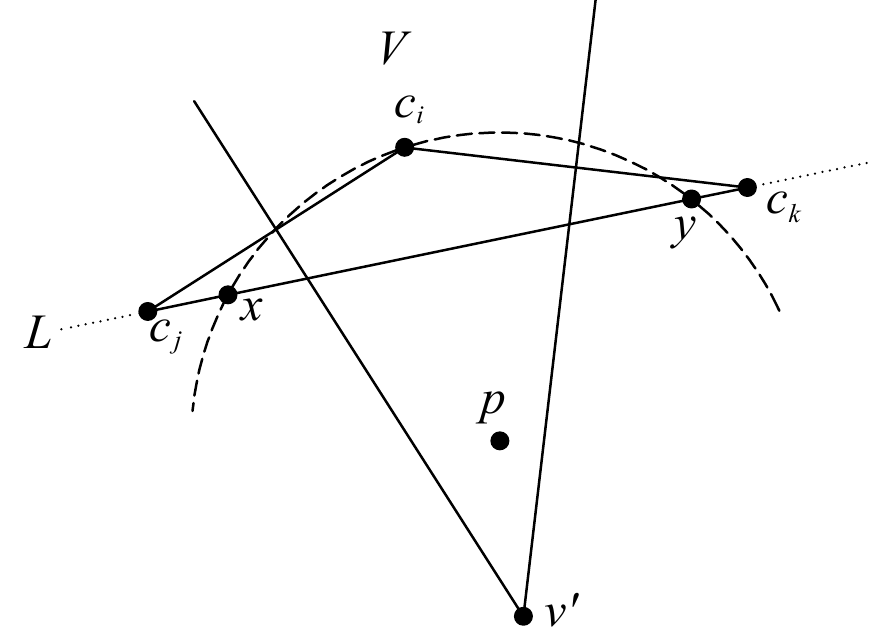}
\caption{An illustration for the proof of Lemma~\ref{lem:Molnarisgood}}
\label{fig:closestpoint}
\end{center}
\end{figure}

If $P$ is type 2, we may apply a similar argument.
\end{proof}

By Lemma~\ref{lem:Molnarisgood} we have that for any cell $P$ in the $M'$-decomposition, if $\F^{\lambda}(P)$ denotes the family of translates of $B^{\lambda}=(1+\lambda)B$ centered at the vertices of $P$ in $C$, then $\rho(\F^{\lambda}(P) | P) = \rho(\F^{\lambda} | P)$.
In other words, to compute the covering ratio of the soft packing in the cell $P$ it suffices to consider the soft disks centered at the vertices of $P$ in $C$.
To finish the proof, we show that for any $M'$-cell $P$, we have $\rho(\F^{\lambda}(P) | P) \leq \rho(\lambda,T_{\alpha})$ for some $\frac{\pi}{6}\leq \alpha \leq \frac{\pi}{4}$.

\emph{Case 1}, $P$ is type 2. To prove the assertion in this case, we need the next lemma.

\begin{Lemma}\label{lem:isosceles}
Let $T=\conv\{p_1,p_2,o\}$ be an isosceles triangle with its apex at $o$, and $2a=|p_1-p_2|$, $b=|p_1-o|$. Assume that $b \geq 1+\lambda$, and let
\[
\rho(a,b)= \frac{\area\left( T \cap \bigcup_{i=1}^2(p_i + B^{\lambda})\right)}{\area(T)}.
\]
Then $\rho(a,b)$ is a strictly decreasing function of both $a$ and $b$.
\end{Lemma}

\begin{proof}[Proof of Lemma~\ref{lem:isosceles}]
Set $\bar{\lambda}=1+\lambda$. Let $a > \bar{\lambda}$. The fact that in this case $\rho(a,b)$ is a (not necessarily strictly) decreasing function of $a$ and $b$ is proved in Lemma in \cite{Andras}. To prove that this function is strictly decreasing, we may apply a straightforward modification of its proof.

Hence, from now on, we assume that $a \leq \bar{\lambda}$.
Then we have
\begin{equation}\label{eq:density2}
\rho(a,b)= \frac{\bar{\lambda}^2 \arccos \frac{a}{b} -\bar{\lambda}^2 \arccos \frac{a}{\bar{\lambda}} + a \sqrt{\bar{\lambda}^2-a^2}}{a \sqrt{b^2-a^2}} .
\end{equation}
This implies that
\[
\rho'_b(a,b)= \frac{b}{b^2-a^2}\left( \frac{\bar{\lambda}^2}{b^2} - \rho(a,b) \right).
\]
Here, using the integral formula for the area of a function given in polar form, it is easy to see that $\rho(a,b)> \frac{\bar{\lambda}}{b} > \frac{\bar{\lambda}^2}{b^2}$, which yields that $\rho'_b(a,b)< 0$.

On the other hand, by an elementary computation, we obtain
\[
\rho'_a(a,b)= \frac{ab^2\sqrt{\bar{\lambda}^2-a^2}-a\bar{\lambda}^2 \sqrt{b^2-a^2}-\bar{\lambda}^2(b^2-2a^2) \left( \arccos \frac{a}{b} -\arccos \frac{a}{\bar{\lambda}} \right)}{a^2 \left( b^2-a^2\right)^{3/2}} .
\]
Let us use the substitutions $b = \frac{a}{\cos \mu}$ and $\bar{\lambda} = \frac{a}{\cos \nu}$. Then we have $0 \leq \nu < \mu < \frac{\pi}{2}$, and
\[
\rho'_a(\mu,\nu) = \frac{\sin 2\nu - \sin 2\mu + 2 \cos 2\mu (\mu-\nu)}{2a \tan^2 \mu \cos^2 \mu \cos^2 \nu} .
\]
Clearly, the denominator of this fraction is positive. On the other hand, it is easy to check that its numerator is a strictly increasing function
of $\nu$ on the interval $[0,\mu]$, and its value is zero if $\nu = \mu$. Thus, we have $\rho'_a(a,b)>0$ for every value of $a$ and $b$.
\end{proof}

Since $P$ is type 2, $P=\cl \left( \conv \{ v,c_i,c_j\} \setminus \conv \{ v',c_i,c_j\} \right)$, for some $c_i,c_j \in C$, where $|v-c_i| = |v-c_j| \geq \sqrt{2}$,
and $|v'-c_i| = |v'-c_j|$.
Let $T = \conv \{ v,c_i,c_j\}$ and $T' = \conv \{ v',c_i,c_j\}$. Then, by Lemma~\ref{lem:isosceles}, we have $\rho(\F^{\lambda} | T) \leq \rho(\F^{\lambda} | T')$, yielding that $\rho(\F^{\lambda} | P) \leq \rho(\F^{\lambda} | T)$. Furthermore, since the legs of $T$ are at least $\sqrt{2}$, and its base is at least $2$, therefore by Lemma~\ref{lem:isosceles} we have $\rho(\F^{\lambda} | T) \leq \rho(\F^{\lambda} | T_0) = \rho\left( \lambda, T_\frac{\pi}{4}\right)$, where $T_0$ is the isosceles right triangle whose hypothenus is of length $2$.

\emph{Case 2}, $P$ is type 1. In this case the sides of $P$ are of length less than $2\sqrt{2}$, and thus, the unit disks centered at the vertices of $P$ are totally separable. This fact is equivalent to the condition that two heights of $P$ are at least two. Hence, the assertion follows immediately from Lemma~\ref{lem:acute}.

\begin{Lemma}\label{lem:acute}
Let $T=\conv\{p_1,p_2,p_3\}$ be an acute triangle with two heights at least two. Let
\[
\rho(\lambda,T) = \frac{\area\left( T \cap \bigcup_{i=1}^3 (p_i + B^{\lambda}) \right)}{\area(T)}.
\]
Then $\rho(\lambda,T) \leq \rho(\lambda,T_{\alpha})$, for some $\frac{\pi}{6} \leq \alpha \leq \frac{\pi}{4}$.
\end{Lemma}

\begin{proof}[Proof of Lemma~\ref{lem:acute}]
First, note that by our assumption, all sidelengths of $T$ are greater than two. Let $R$ be the circumradius of $T$.

If $R \leq 1+\lambda$, then the assertion follows by a simple geometric observation.
Thus, we prove the statement under the condition that $R > 1+\lambda$. This condition implies that the circumcenter of $T$ is not covered by the three soft disks, and also that at most two of the three soft disks intersect.
If there is a soft disk that does not intersect the other two soft disks (i.e. two sides of $T$ are longer than $2 + 2 \lambda$), we may move it towards the opposite side of $T$, and, thus, increase the covering ratio. Hence, if $\rho(\lambda,T)$ is maximal, then it has at most one side longer than $2+2\lambda$.
Let $p_1$ be the vertex of $T$ such that the altitude starting at $p_1$ is a shortest altitude. Then, by the area formula for triangles, if $T$ has a side longer than $2+2\lambda$, then it is $[p_2,p_3]$, and hence, $p_1 + B^{\lambda}$ intersects the other two soft disks.

Let $T' = \conv\{p'_1,p_2,p_3\}$ be an isosceles triangle with circumradius $R$. In the remaining part of the proof we show that $\rho(\lambda,T') \geq \rho(\lambda,T)$.
Observe that $T'$ also has two (equal) heights that are at least two, and also that its base is not shorter than its legs.
Thus, this inequality implies the assertion of the lemma, since, if these two heights of $T'$ are greater than two, then we can replace $T'$ by a smaller similar copy of itself, which clearly increases its covering ratio.

Let $o$ be the circumcenter of $T$, and for $i=1,2,3$, let $t_i = \area (\conv \{ o,p_j,p_k\})$, and 
\[
t^{tr}_i = \area (\left( \left( \left( p_j + (1+\lambda) \BB\right) \cup \left( p_k + (1+\lambda) \BB\right) \right) \cap \conv \{ o,p_j,p_k\} \right),
\]
where $\{i,j,k\} = \{ 1,2,3 \}$.
We define $t'_i$ and $t'^{tr}_i$ similarly for the triangle $T'$.
Note that by Lemma~\ref{lem:isosceles}, we have $\frac{t^{tr}_1}{t_1} = \frac{t'^{tr}_1}{t'_1} \leq \min\{ \frac{t^{tr}_2}{t_2} , \frac{t^{tr}_3}{t_3}\}$.
Since, clearly, $t_2+t_3 \leq 2 t'_2 = 2t'_3$, it is sufficient to prove that $\frac{t^{tr}_2+t^{tr}_3}{t_2+t_3} \leq \frac{t'^{tr}_2}{t'_2}$.

Set $\mu = p_1p_2p_3\angle$, $\nu = p_1p_3p_2 \angle$, and $\bar{\mu}  = \frac{\mu+\nu}{2} = p'_1p_2p_3\angle =
p'_1p_3p_2\angle$, and $f(x)= \frac{\pi}{2}-x-\arccos(R \sin x)+r \sin (x) \sqrt{1-R^2 \sin^2 x}$, and $g(x) = \frac{1}{2} R^2 \sin (2x)$.
Then the desired inequality can be written in the form
\[
\frac{f(\mu)+f(2\tau-\mu)}{g(\mu)+g(2\tau-\mu)} \leq \frac{f(\tau)}{g(\tau)}
\]
for some $0 < \mu \leq \tau \leq 2\tau-\mu < \arcsin \frac{1}{R} < \frac{\pi}{2}$.
Let us define
\[
F(\mu,\tau) = 2f(\tau) \left( g(\mu) + g(2\tau-\mu)\right) - 2 g(\tau) \left( f(\mu)+f(2\tau-\mu) \right).
\]
Note that $F(\tau,\tau) = 0$ for every value of $\tau$. We show that $F(\mu,\tau)$ is a strictly decreasing function of $\mu$ for every value of $\tau$, which readily implies the assertion.
By an elementary computation, we obtain
\[
F'_{\mu}(\mu,\tau) = 2R^2 \sin2\tau \left( 2 f(\tau) \sin (2\tau-2\mu) - R \cos \alpha \sqrt{1-R^2 \sin^2 \mu} + R \cos(2\tau-\mu) \sqrt{1-R^2 \sin^2(2\tau-\mu)}\right).
\]
Using the inequalities $2R^2 \sin 2\tau > 0$, $\sin(2\tau-2\mu) > 0$, $f(\tau) < g(\tau)$, and some trigonometric identities, we obtain that
\[
F'_{\mu}(\mu,\tau) < 2R^2 \sin2\tau \left( h(\mu) - h(2\tau-\mu) \right),
\]
where $h(x)=R^2 \cos^2 x - R \cos x \sqrt{1-R^2 \sin^2 x}$. Observe that
\[
h'(x) = \frac{R \cos x \left( R \cos x - \sqrt{1-R^2 \sin^2 x} \right)^2}{\sqrt{1-R^2 \sin^2 x}} > 0
\]
if $0 < x < \arcsin \frac{1}{R}$. This implies that $h(\mu) < h(2\tau-\mu)$, from which the inequality $F'_{\mu}(\mu,\tau) < 0$ readily follows.
\end{proof}

This completes the proof of Theorem~\ref{thm:soft_packing}.                                  \end{proof}

\begin{figure}[ht]
    \centering
    \begin{minipage}{.45\textwidth}
        \begin{center}
\includegraphics[width=0.9\textwidth]{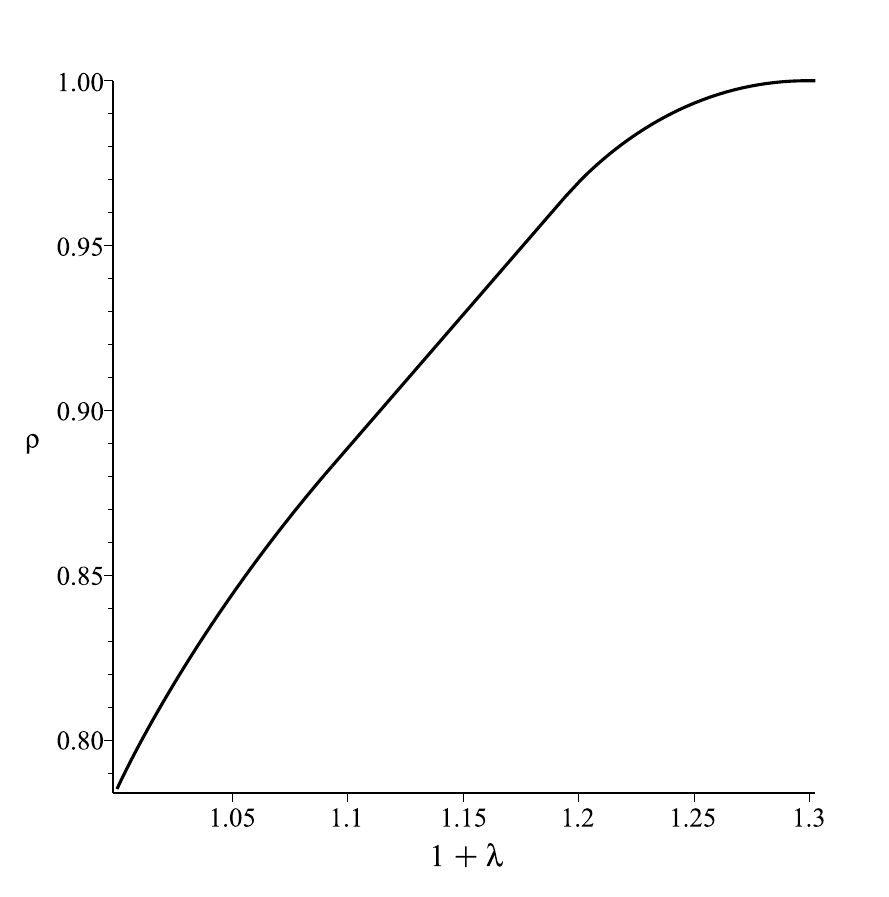}
\end{center}
    \end{minipage}%
    \begin{minipage}{0.45\textwidth}
        \begin{center}
\includegraphics[width=0.9\textwidth]{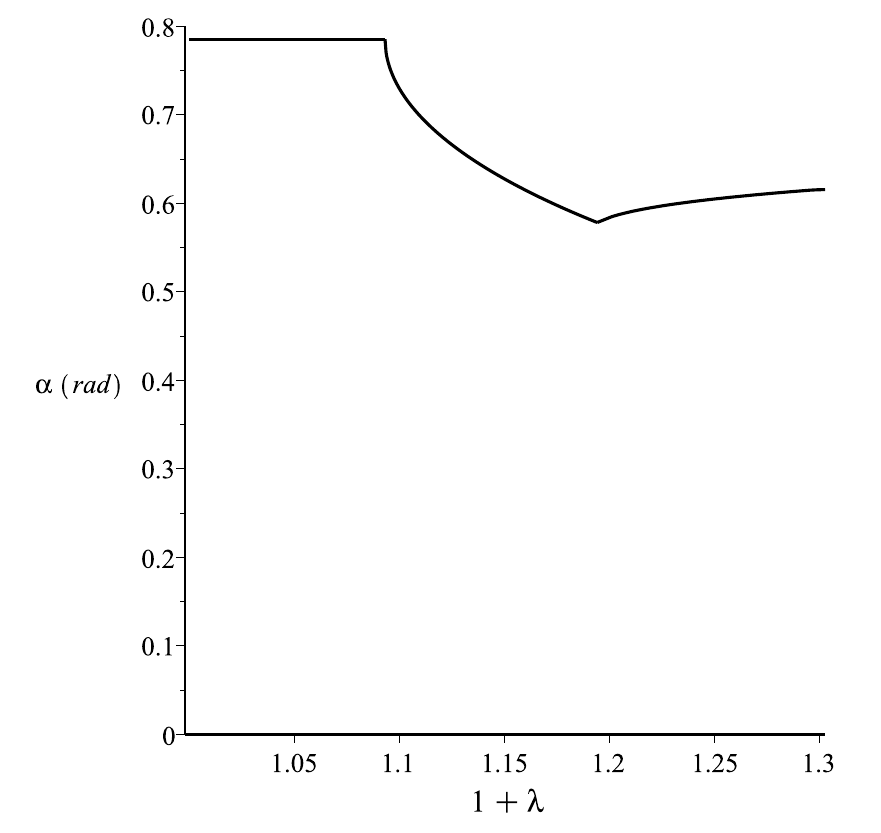}
\end{center}
    \end{minipage}
		\caption{Maximal covering ratios of soft circle packings (on the left), and half-angles of isosceles triangles with maximal covering ratios (on the right)}
		\label{fig:density}
\end{figure}

\begin{Corollary}
An elementary computaion yields that the smallest value $\Lambda$ of $\lambda$ with the property that $\rho(\lambda,T_{\alpha}) = 1$ for some value of $\alpha$ is 
$\Lambda = \frac{3\sqrt{3}}{4} - 1 \approx 0.299$. Thus, by Theorem~\ref{thm:soft_packing}, $\rho_{\lambda,B}^{sep}= 1$ if, and only if $\lambda \geq \frac{3\sqrt{3}}{4} - 1$. Furthermore, $\rho(\Lambda,T_{\alpha}) = 1$ is satisfied if, and only if $\alpha = \arccos \sqrt{\frac{2}{3}} \approx 35.27644^{\circ}$. It is worth rephrasing this result in terms of closeness of packings introduced by L. Fejes T\'oth (\cite{FTL78}) as follows. If $\cal P$ be a packing of unit disks in $\Eu^2$, then let $r_{sep}({\cal P})$ be the supremum of $r>0$ for which there exists a circular disk of radius $r$ having no point in common with the elements of $\cal P$. Then call $r_{sep}({\cal P})$ the {\it closeness} of $\cal P$. Thus, the above mentioned claim can be stated saying that $ r_{sep}({\cal P})\geq  \frac{3\sqrt{3}}{4} - 1$ holds for any totally separable unit disk packing ${\cal P}$ of $\Eu^2$ with equality for a unique totally separable lattice packing.
\end{Corollary}

\begin{Remark}\label{rem:max_density}
By Theorem~\ref{thm:soft_packing}, we have that $\rho_{\lambda,\BB}^{sep} = \rho(\lambda,T_{\alpha})$ for some $\frac{\pi}{6} \leq \alpha \leq \frac{\pi}{4}$.
Unfortunately, in general the value(s) of $\alpha$ where $\rho(\lambda,T_{\alpha})$ is maximal for a fixed value of $\lambda$ can be computed only numerically. The graph on the left in Figure~\ref{fig:density} shows the maximal values of $\rho(\lambda,T_{\alpha})$ as a function of $\lambda$, and the one on the right fshows values of $\alpha$, belonging to these covering ratios. Let $t_0$ denote the unique solution of the equation $\sin 2t = \frac{\pi}{2}-2t$, and set $\lambda_1 = \frac{1}{\cos t_0}-1 \approx 0.093$ and $\lambda_2 = \frac{2 \cos t_0}{\sqrt{4\cos^2 t_0-1}} - 1 \approx 0.194$. If $0 < \lambda \leq \lambda_2$, then there is a unique optimal value of $\alpha$, and in the corresponding triangle $T_{\alpha}$ the two soft disks centered at the vertices of the base do not overlap. Within this case, if $ < \lambda \leq \lambda_1$, then the covering ratio is maximal for $\alpha = \frac{\pi}{4}$. If $\lambda_1 \leq \lambda \leq \lambda_2$, then the optimal value of $\alpha$ is $\alpha = \frac{1}{2} \arcsin \left( (1+\lambda) \cos t_ 0 \right)$, and the maximal covering ratio is the linear function $ (1+\lambda) \left( \frac{\pi}{4 \cos t_0} - \frac{t_0}{\cos t_0} + \sin t_ 0 \right)$. If $\lambda > \lambda_2$, then the two soft circles centered at the endpoints of the base overlap, we could express the maximal values of the covering ratios only numerically. 
\end{Remark}

\bigskip

\section{Appendix}

We cannot resist to raise the following questions motivated by the theorems proved in this paper.

\begin{Problem}
Characterize the case of equality in (\ref{eq:Oler}) of Theorem~\ref{thm:Oler}.
\end{Problem}

\begin{Definition}
For any $o$-symmetric convex domain $K$ in $\Eu^2$, $n>1$, and $\lambda\ge 0$ let

\[
a_{sep}(K,n,\lambda) = \min \left\{\area \left( \bigcup_{i=1}^n (c_i+ K^{\lambda})\right) : \{c_1 + K^{\lambda}, \ldots, c_n + K^{\lambda} \} \in \P^{sep}_{K,n,\lambda} \right\}.
\]
\end{Definition}

\begin{Problem}\label{soft-Betke}
Compute $a_{sep}(K,n,\lambda)$ for given $o$-symmetric convex domain $K$, $n>1$, and $\lambda\ge 0$.
\end{Problem}

If $K$ is an $o$-symmetric convex domain in $\Eu^2$, $n>1$, and $\{c_1,c_2,\dots , c_n\}\subseteq \Eu^2$, then it is easy to see that
\begin{equation}\label{large}
\area\left( \bigcup_{i=1}^n (c_i + K^{\lambda})\right) = (1+\lambda)^2 \area(K)+2\lambda A(\conv \{c_1,c_2,\ldots,c_n \}, K)+o(1+\lambda).
\end{equation}
for $(1+\lambda)\to+\infty$.

Based on (\ref{large}) the following is immediate from $(\ref{lem:Radon}.2)$ of Lemma~\ref{lem:Radon}.

\begin{Remark}
\[
a_{sep}(K,n,\lambda) \geq
\]
\[
(1+\lambda)^2 \area(K) + \lambda \frac{\area(\square(K))}{4} \min_{ \{c_i + K : i=1,2,\ldots,n \} \in \P^{sep}_{K,n,0}} \left\{ M_{K} (\bd (\conv \{ c_1,\ldots, c_n\}))\right\}
 + o(1+\lambda)
\]
for $(1+\lambda)\to+\infty$.
\end{Remark}

Thus, the problem of lower bounding $a_{sep}(K,n,\lambda)$ for large $\lambda$ leads us to

\begin{Problem}
For a given $o$-symmetric convex domain $K$ in $\Eu^2$ and given $n > 1$ compute 
\[
 \min_{ \{c_i + K : i=1,2,\ldots,n \} \in \P^{sep}_{K,n,0}} \left\{ M_{K} (\bd (\conv \{ c_1,\ldots, c_n\}))\right\}=
 \]
 \[
 \min_{ \{c_i + K : i=1,2,\ldots,n \} \in \P^{sep}_{K,n,0}} \left\{ M_{K} \left(\bd \left(  \conv \left( \bigcup_{i=1}^n (c_i+K) \right)  \right)\right)\right\}-M_K(\bd K).
 \]
\end{Problem}

Recall that the {\it maximum separable contact number} $c_{sep}(K, n, 2)$ is the largest contact number of totally separable packings of $n$ translates of a given ($o$-symmetric) convex domain $K$ in $\Eu^2$ for given $n>1$, where the contact number of a packing is simply the number of touching pairs among the packing elements.

\begin{Remark}
Let $n>1$ be given and let $K$ be an $o$-symmetric convex domain in $\Eu^2$. Then it is easy to see that there exists $\lambda(K, n)>0$  with the following property: for any $\lambda$ with $0\leq\lambda\leq \lambda(K, n)$ and any $\{ c_1 + K^{\lambda}, c_2 + K^{\lambda}, \ldots, c_n + K^{\lambda} \} \in \P^{sep}_{K,n,\lambda}$ the number of pairs
$\{ c_i + K^{\lambda}, c_j + K^{\lambda}\}$ with $(c_i + K^{\lambda})\cap (c_j + K^{\lambda})\neq\emptyset$ is at most $c_{sep}(K, n, 2)$. Furthermore, if $K$ is smooth, then there exits $\lambda^*(K, n)>0$ with $\lambda^*(K, n)\leq \lambda(K, n)$ such that no three of the sets $\{ c_1 + K^{\lambda}, c_2 + K^{\lambda}, \ldots, c_n + K^{\lambda} \}$
intersect. As a result it is not hard to see that
\[
n\area(K^\lambda)- c_{sep}(K, n, 2)A_{\rm max}(K, \lambda)\leq a_{sep}(K,n,\lambda) \leq n\area (K^\lambda)- c_{sep}(K, n, 2)A_{min}(K, \lambda),
\]
where $A_{\rm min}(K, \lambda)=\min\{\area \left((c_i + K^{\lambda})\cap (c_j + K^{\lambda})\right)\ |\ |c_i-c_j|_K=2\}$ 
\newline
(resp., $A_{\rm max}(K, \lambda)=$ $\max\{\area \left((c_i + K^{\lambda})\cap (c_j + K^{\lambda})\right)\ |\ |c_i-c_j|_K=2\}$).
\end{Remark}

Thus, the problem of bounding $a_{sep}(K,n,\lambda)$ for small $\lambda$ leads us to

\begin{Problem}\label{separable-contact-numbers}
Let $K$ be a (smooth) convex domain in $\Eu^2$ and $n>1$. Then compute 
\[
c_{sep}(K,n,2).
\]
\end{Problem}

In connection with Problem~\ref{separable-contact-numbers} the following result was proved in \cite{BeKhOl}.

\begin{Theorem}
\item{(A)} $c_{sep}(K,n,2) = \left\lfloor 2n - 2\sqrt{n}\right\rfloor$, for any smooth strictly convex domain $K$ in $\Eu^2$.
\item{(B)} Let $R$ be a smooth Radon domain and let $n=\ell(\ell+\epsilon)+k\ge 4$ be the unique decomposition of a positive integer $n$ such that $k$, $\ell$ and $\epsilon$ are integers satisfying $\epsilon\in \{0,1\}$ and $0\le k< \ell+\epsilon$. Suppose that $\cal{P}$ is a totally separable packing of $n$ translates of $R$ with $c_{sep}(R,n,2) =\lfloor{2n-2\sqrt{n}}\rfloor$ contacts. If $k\ne 1$, then $\cal{P}$ is a finite lattice packing lying on an Auerbach lattice of $R$, while if $k=1$, then all but at most one translate in $\cal{P}$ form a lattice packing on an Auerbach lattice of $R$.
\end{Theorem}

\begin{Definition}
Let $\F = \{ c_i + K : c_i\in\Eu^2\ {\text for}\ i \in \N\}$ be a totally separable packing (resp., lattice packing) of translates of the $o$-symmetric convex domain $K$ in $\Eu^2$. Then $\F^{\lambda} = \{ c_i + (1+\lambda)K : i \in \N\}$
is called a \emph{totally separable soft packing} (resp.,  \emph{totally separable soft lattice packing}) of translates of the \emph{soft convex domain}  $K^{\lambda}=(1+\lambda)K$ with \emph{soft parameter} $\lambda > 0$.
In this case the \emph{(upper) covering ratio} of the soft packing $\F^{\lambda}$ is defined as 
\[
\rho(\F^\lambda) = \limsup_{r \to \infty} \frac{\area \left( r K \cap \bigcup_{i\in\N } (c_i+K^{\lambda})\right)}{\area (r K)}.
\]
We denote by $\rho_{\lambda, K}^{sep}$ (respectively, $\rho_{\lambda, K}^{sep, lattice}$) the supremum of the (upper) covering ratios over the family of totally separable soft packings (respectively, of totally separable soft lattice packings) of translates of the soft convex domain $K^{\lambda}$ with soft parameter $\lambda > 0$.
\end{Definition}

\begin{Problem}
Let $K$ be an $o$-symmetric convex domain in $\Eu^2$. Then prove or disprove that $\rho_{\lambda,K}^{sep}=\rho_{\lambda,K}^{sep,lattice} $
holds for every $\lambda>0$.
\end{Problem}

\begin{Definition}
If $\cal P$ is a totally separable packing of translates of an $o$-symmetric convex domain $K$ in $\Eu^2$, then let $r_{sep}({\cal P})$ be the supremum of $r>0$ for which there exists a translate of $rK$ having no point in common with the elements of $\cal P$. Then call $r_{sep}({\cal P})$ the {\it closeness} of $\cal P$ and set 
\[
\underline{r}_{sep}(K)=\inf\{r_{sep}({\cal P})\ :\ {\cal P}\ \text{is a totally separable packing of translates of}\ K\ {\text in}\ \Eu^2\}.
\]
\end{Definition}

\begin{Problem}
Prove or disprove that for any $o$-symmetric convex domain $K$ of $\Eu^2$  we have $\underline{r}_{sep}(K)=r_{sep}({\cal P})$ for some totally separable lattice packing $\cal P$ of $K$.
\end{Problem}

\vskip1.0cm

\noindent K\'aroly Bezdek \\
\small{Department of Mathematics and Statistics, University of Calgary, Calgary, Canada}\\
\small{Department of Mathematics, University of Pannonia, Veszpr\'em, Hungary}\\
\small{\texttt{bezdek@math.ucalgary.ca}}

\bigskip

\noindent and

\bigskip

\noindent Zsolt L\'angi \\
\small{MTA-BME Morphodynamics Research Group and Department of Geometry}\\ 
\small{Budapest University of Technology and Economics, Budapest, Hungary}\\
\small{\texttt{zlangi@math.bme.hu}}

\end{document}